\documentclass[12pt,a4paper]{article}
\usepackage{amsfonts,amssymb}
\usepackage{latexsym}
\usepackage[usenames,dvipsnames]{color}
\usepackage{subfigure}
\usepackage{amsmath}
\usepackage{amsthm}
\usepackage{fullpage}
\usepackage{hyperref}

\newtheorem{thm}{Theorem}[section]

\newtheorem{lem}[thm]{Lemma}
\newtheorem{cor}[thm]{Corollary}

\newcommand{\Z}{\mathbb{Z}}

\newcommand{\gl}{\mathop{\rm GL}\nolimits}

\newcommand{\rest}{\mathbin{\upharpoonright}}

\newcommand{\la}{\lambda}

\begin{document}

\title{The Spectrum of Group-Based Complete Latin Squares}

\author{M.~A.~Ollis\footnote{Corresponding author, email address: \texttt{matt@marlboro.edu.}} \  and Christopher R. Tripp   \\
    {\it Marlboro College, P.O.~Box A, Marlboro,} \\    
    {\it Vermont 05344, USA.} }  
 
\date{}

\maketitle

\begin{abstract}
We construct sequencings for many groups that are a semi-direct product of an odd-order abelian group and a cyclic group of odd prime order.  It follows from these constructions that there is a group-based complete Latin square of order~$n$ if and only if~$n \in \{ 1,2,4\}$ or there is a non-abelian group of order~$n$.

\vspace{3mm}
\noindent
{\bf Keywords:} complete Latin square; directed terrace; Keedwell's Conjecture; sequencing.  \\
{\bf MSC2010:} 05B15

\end{abstract}

\section{Introduction}\label{sec:intro}

A {\em Latin square} of order~$n$ is an $n \times n$ array of symbols from a set of size~$n$ with each symbol appearing once in each row and once in each column.  A Latin square is {\em row-complete} or {\em Roman} if each pair of distinct symbols appears in adjacent positions in a row once in each order.  It is {\em complete} if both it and its transpose are row-complete.

Interest in complete and row-complete Latin squares was originally prompted by their usefulness in the design of experiments where neighboring treatments, whether in space or time, might interact.  See, for example,~\cite{BJ08}.

The Cayley table of a finite group of order~$n$ is a Latin square.  The principal question for this work is to determine at what orders there is a group for which it is possible to permute the rows and columns of its Cayley table to give a complete Latin square.  To this end, consider the following definition.  Let~$G$ be a group of order~$n$ and ${\bf a} = (a_1, a_2, \ldots, a_n)$ be an arrangement of the elements of~$G$.  Define~${\bf b} = (b_1, b_2, \ldots, b_{n-1})$ by~$b_i = a_i^{-1}a_{i+1}$ for each~$i$.  If~${\bf b}$ includes each non-identity element of~$G$ exactly once then~${\bf b}$ is a {\em sequencing} of~$G$ and~${\bf a}$ is a {\em directed terrace} for $G$.  Call a group that admits a sequencing {\em sequenceable}.

\begin{thm}\label{th:gbcls}{\rm \cite{Gordon61}}
The rows and columns of the Cayley table of a group~$G$ of order~$n$ may be permuted to give a complete Latin square if and only if~$G$ is sequenceable.
\end{thm}

\begin{proof}[Proof idea]
Let $(g_1, \ldots, g_n)$ and $(h_1, \ldots, h_n)$ be arrangements of the elements of~$G$ and let the $(i,j)$-entry of a Cayley table of~$G$ be given by~$g_ih_j$.  Then it is a complete Latin square if and only if $(g_1^{-1}, g_2^{-1}, \ldots, g_n^{-1})$ and $(h_1, h_2, \ldots, h_{n})$ are both directed terraces for~$G$.
\end{proof}

Let~$\Z_{n} = \{0 ,1, \ldots, n-1 \}$ be the additively written cyclic group of order~$n$.  When~$n$ is even,
$$ (0, n-1, 1, n-2, 2, n-3, \ldots, n/2) $$
is a directed terrace for~$\Z_n$.  The first use of a ``zig-zag" construction of this type is Walecki's in 1892~\cite{Lucas92}; see~\cite{Alspach08} for more of its history.  The first use of it to control neighbor balance in Latin squares seems to be due to Williams in 1949~\cite{Williams49}. 

The systematic consideration of sequenceability for arbitrary groups was initiated by Gordon~\cite{Gordon61} where, as well as proving Theorem~\ref{th:gbcls}, it is shown that an abelian group is sequenceable if and only if it has exactly one involution.  Hence if~$n$ is even there is a group-based complete Latin square of order~$n$.  The result also implies that for odd orders we must turn our attention to non-abelian groups.  Several families of groups of odd order are known to be sequenceable, including: a group of order~$p^m$ for each odd prime~$p$ and each~$m \geq 3$~\cite{Wang02}; many groups of order~$pq$ for distinct primes~$p$ and~$q$~\cite{Keedwell81, Wang02}; and a group of each order~$3m$ where~$m$ is powerful (i.e.~for each prime~$p$ dividing~$m$, $p^2$ also divides~$m$)~\cite{Ollis14}.  

We construct sequencings for some semi-direct products $\Z_q \ltimes A$ where~$A$ is an abelian group of odd order and~$q$ is an odd prime, including all possible such groups when~$A$ is cyclic.  These constructions allow us to determine the full spectrum of orders at which a group-based complete Latin square exists:
\begin{thm}\label{th:spectrum}
There is a group-based complete Latin square of order~$n$ if and only if~$n=1$, $n$ is even, or there exists a non-abelian group of order~$n$.  That is, if and only if~$n=1$, $n$ is even, or~$n$ has either a prime divisor~$p$ with~$p^3 | n$ or a prime-power divisor~$p^k$ such that~$p^k \equiv 1 \pmod{q}$ for some prime divisor~$q$ of~$n$. 
\end{thm}

This result also gives the spectrum for group-based row-complete Latin squares.  However, whereas all known complete Latin squares are group-based, there are alternative methods known for constructing row-complete squares.  Row-complete Latin squares are known to exist at orders~1 and~2 and at every composite order~\cite{Higham98, Williams49}.  They are known not to exist at order~3,~5 or~7.  The question remains open at other odd primes.

On the question of which groups are sequenceable, the abelian case is settled as mentioned above, and the three non-abelian groups of orders~6 and~8 are not sequenceable.  {\em Keedwell's Conjecture} is that all other non-abelian groups are sequenceable. In addition to those already mentioned, groups known to satisfy Keedwell's Conjecture include dihedral groups~\cite{Isbell90, Li97}, soluble groups with a single involution~\cite{AI93}, and groups of order at most~255~\cite{OllisTFSG}.  See~\cite{OllisSurvey} for a survey of this and related problems. 

In the next section we develop a framework for constructing sequencings of groups of the form~$G = \Z_q \ltimes A$ for prime~$q$ and odd-order abelian~$A$. In Section~\ref{sec:dt} we show that the framework can be successfully completed for a variety of choices of~$A$, whence Theorem~\ref{th:spectrum}.

\section{The construction}\label{sec:constr}

Let~$A$ be an abelian group of order~$m$ with an automorphism~$\alpha$ of prime order~$q$.  Let
$$G = \Z_q \ltimes_{\alpha} A =  \{ (u,v) : u \in \Z_q, \ v \in A\}, \hspace{6mm} (u,v)(x,y) = (u+x, \alpha^x(v) + y).$$
This is a group of order~$n = mq$. 
Let~$\la$ be a primitive root of~$q$ such that~$\la / (\la -1)$ is also a primitive root; Wang~\cite{Wang02} shows that the existence of such a~$\la$ follows from results of~\cite{CM91}.

Here is the template we use for a proposed directed terrace of~$G$:
$$
\begin{array}{l}
(0,g_1), (0,g_2), \ldots, (0, g_t), \\
(1, h_{11}), (\la^{q-2}, h_{21}), (\la^{q-3}, h_{31}), \ldots, (\la, h_{q-1,1}), \\
(1, h_{12}), (\la^{q-2}, h_{22}), (\la^{q-3}, h_{32}), \ldots, (\la, h_{q-1,2}), \\
\hspace{10mm} \vdots \\
(1, h_{1,m-1}), (\la^{q-2}, h_{2,m-1}), (\la^{q-3}, h_{3,m-1}), \ldots, (\la, h_{q-1,m-1}), \\
(\la, 0), (\la^2/(\la-1), 0), (\la^3/(\la-1)^2, 0), \ldots, (\la^{q-2}/(\la-1)^{q-3}, 0), (\la -1, 0), \\
(0, g_{t+1}), (0, g_{t+2}), \ldots, (0,g_m).
\end{array}
$$

This generalizes the structures of directed terraces in~\cite{Ollis14} and~\cite{Wang02} (each of which grew out of that of~\cite{Keedwell81}).

The corresponding sequencing (terms immediately after a semi-colon correspond to quotients formed from terms on different rows):
$$
\begin{array}{l}
(0,g_2-g_1), (0,g_3-g_2), \ldots, (0, g_t - g_{t-1}); (1, h_{11} - \alpha(a_t)), \\
(\la^{q-2} - 1, h_{21} - \alpha^{\la^{q-2}-1}(h_{11}) ), (\la^{q-3} - \la^{q-2}, h_{31} - \alpha^{\la^{q-3} - \la^{q-2}}(h_{21}) ), \ldots, \\ 
\hspace{10mm} (\la - \la^2,h_{q-1,1} - \alpha^{\la - \la^2}(h_{q-2,1})   ) ;  (1-\la, h_{12} - \alpha^{1-\la}(h_{q-1,1})    ), \\
(\la^{q-2} - 1, h_{22} - \alpha^{\la^{q-2}-1}(h_{12}) ), (\la^{q-3} - \la^{q-2}, h_{32} - \alpha^{\la^{q-3} - \la^{q-2}}(h_{22}) ), \ldots, \\ 
\hspace{10mm} (\la - \la^2,h_{q-1,2} - \alpha^{\la - \la^2}(h_{q-2,2})   ) ;  (1-\la, h_{13} - \alpha^{1-\la}(h_{q-1,2})    ), \\
\hspace{10mm} \vdots \\
(\la^{q-2} - 1, h_{2,m-1} - \alpha^{\la^{q-2}-1}(h_{1,m-1}) ), (\la^{q-3} - \la^{q-2}, h_{3,m-1} - \alpha^{\la^{q-3} - \la^{q-2}}(h_{2,m-1}) ), \ldots, \\ 
\hspace{10mm} (\la - \la^2,h_{q-1,m-1} - \alpha^{\la - \la^2}(h_{q-2,m-1})   ) ;  (0, -h_{q-1,m-1})    , \\
( \la^2/(\la-1) - \la ,0), (\la^3/(\la-1)^2 - \la^2/(\la-1)  ,0), \ldots, 
\\ \hspace{10mm}  ( \la -1 - \la^{q-2}/(\la-1)^{q-3}, 0 ); (-(\la -1), g_{t+1} ),  \\
(0,g_{t+2}-g_{t+1}), (0,g_{t+3}-g_{t+2}), \ldots, (0, g_m - g_{m-1}).
\end{array}
$$

What requirements must we satsify to ensure these are a directed terrace and sequencing?

Observe that $$\la^2/(\la-1) - \la, \la^3/(\la-1)^2 - \la^2/(\la-1), \ldots,  \la -1 - \la^{q-2}/(\la-1)^{q-3}$$ lists all of the elements~$\Z_q \setminus \{ 0,1 \}$.   We therefore require that $h_{11} - \alpha(g_t) = 0$ so that every element of the form~$(x,0)$, for~$x \in \Z_q \setminus \{ 0 \}$, appears in the sequencing.

We need that $g_1, g_2, \ldots g_m$ includes every element of~$A$ exactly once and that each of the following sequences list all elements of $A \setminus \{ 0 \}$:
$$\begin{array}{l}
 h_{11}, h_{12}, \ldots, h_{1,m-1} \\
 h_{21}, h_{22}, \ldots, h_{2,m-1} \\
\hspace{10mm} \vdots \\
 h_{q-1,1}, h_{q-1,2}, \ldots, h_{q-1,m-1} \\
g_2 - g_1, g_3 - g_2, \ldots, g_t - g_{t-1},  -h_{q-1,m-1}, g_{t+2} - g_{t+1},  \ldots, g_m - g_{m-1} \\
h_{21} - \alpha^{\la^{q-2}-1}(h_{11}) , h_{22} - \alpha^{\la^{q-2}-1}(h_{12}), \ldots, h_{2,m-1} - \alpha^{\la^{q-2}-1}(h_{1,m-1})  \\
h_{31} - \alpha^{\la^{q-3} - \la^{q-2}}(h_{21}), h_{32} - \alpha^{\la^{q-3} - \la^{q-2}}(h_{22}), \ldots,  h_{3,m-1} - \alpha^{\la^{q-3} - \la^{q-2}}(h_{2,m-1}) \\
\hspace{10mm} \vdots \\
h_{q-1,1} - \alpha^{\la - \la^2}(h_{q-2,1}), h_{q-1,2} - \alpha^{\la - \la^2}(h_{q-2,2}), \ldots, h_{q-1,m-1} - \alpha^{\la - \la^2}(h_{q-2,m-1}) \\
h_{12} - \alpha^{1-\la}(h_{q-1,1}), h_{13} - \alpha^{1-\la}(h_{q-1,2}), \ldots, h_{1,m-1} - \alpha^{1-\la}(h_{q-1,m-2}) ,  g_{t+1}
\end{array}$$
  
In order to meet these requirements, we introduce two auxilliary types of sequence.  
Let~${\bf a} = (a_1, a_2, \ldots, a_{m-1})$ be an arrangement of the non-identity elements of~$A$ and define~${\bf b} = (b_1, b_2, \ldots, b_{m-1})$ by~$b_i = a_{i+1} - a_i$, where the indices are calculated modulo~$m-1$ (so $b_{m-1} = a_1 - a_{m-1}$).  If~${\bf b}$ also contains all of the non-identity elements of~$A$, then ${\bf b}$ is a {\em rotational sequencing} or {\em R-sequencing} of~$A$ and ${\bf a}$ is the associated {\em directed rotational terrace} or {\em directed R-terrace}.

Similarly, let~${\bf c} = (c_1, c_2, \ldots, c_{m-1})$ be an arrangement of the non-identity elements of~$A$ and define~${\bf d} = (d_1, d_2, \ldots, d_{m-1})$ by~$d_i = c_{i} + c_{i+1}$, where, again, the indices are calculated modulo~$m-1$ (so $d_{m-1} = c_{m-1} + c_{1}$).  If~${\bf d}$ also contains all of the non-identity elements of~$A$, then ${\bf a}$ is a {\em \#-harmonious sequence} and~$A$ is {\em \#-harmonious}.

Both of these objects are of interest in their own right.  For example, the existence of either a rotational sequencing or a \#-harmonious sequence for a group implies that that group has a complete mapping and hence its Cayley table has an othogonal mate.  Rotational sequencings were introduced by Ringel~\cite{Ringel74} and studied extensively in~\cite{FGM78}.  They have various different but equivalent formulations in the literature.  \#-harmonious sequences were introduced and studied in~\cite{BGHJ91} and were studied further and named in~\cite{Wang07}.  In each case, the existence question is completely settled in abelian groups (each of them may also be studied in non-abelian groups, which is not relevant here):

\begin{thm}\label{th:hash_rseq}{\rm \cite{AKP17,BGHJ91,FGM78}}
An abelian group has a rotational sequencing if and only if it does not have exactly one involution.  An abelian group has a \#-harmonious sequence if and only if it is not~$\Z_3$ and does not have exactly one involution.
\end{thm}

We shall see some aspects of the known construction methods for rotational sequencings and \#-harmonious sequences in the next section.    Here we see how to combine them to make the task of completing the template successfully more tractable:

\begin{thm}\label{th:constr}
Let~$A$ be an abelian group of odd order~$m$ with an automorphism~$\alpha$ of odd prime order~$q$ and let~$\la$ be a primitive root of~$q$ such that~$\la / (\la-1)$ is also a primitive root of~$q$.
Let  ${\bf a} = (a_1, a_2, \ldots, a_{m-1})$ be a directed  rotational terrace and ${\bf c} = (c_1, c_2, \ldots, c_{m-1})$ be a \#-harmonious sequence for~$A$.  If $a_1 = c_1 + c_{m-1} - \alpha^{q-1}(c_1)$ and $a_{m-1} = a_1 - \alpha^{\la-1}(c_{m-1})$ then $G = \Z_q \ltimes_{\alpha} A$ is sequenceable.
\end{thm}

\begin{proof}
First, we use~${\bf a}$ and~${\bf c}$ to assign values to the~$g_i$ and~$h_{ij}$ in the template.  We shall see that the difference/sum properties they have, combined with the extra conditions, allow us to satisfy all of the constraints.

Let~$g_1 = \alpha^{q-1}(c_1)$ and let $g_i = a_{i-1} + \alpha^{q-1}(c_1)$ for~$i>1$.  As~${\bf a}$ is a directed rotational terrace the sequence $g_1, g_2, \ldots g_m$ includes every element of~$A$ once.  

For odd~$i$ set $h_{ij} = c_j$; for even~$i$ set~$h_{ij} = -  \alpha^{\la -1}(c_j)$.  Set~$t=1$.  This gives 
$$h_{11} - \alpha(g_t) = c_1  - \alpha( \alpha^{q-1}(c_1) )  = 0$$
as required.  We now need to consider the sequences that are required to contain all of the non-zero elements of~$A$.  With the assignments of the~$g_i$, $h_{ij}$ and~$t$, these become:
$$\begin{array}{l}
 c_{1}, c_{2}, \ldots, c_{m-1} \\
\alpha^{\la -1}(c_1), \alpha^{\la -1}(c_2), \ldots, \alpha^{\la -1}(c_{m-1}) \\
 \alpha^{\la -1}(c_{m-1}) , a_{2} - a_{1}, a_3 - a_2,  \ldots, a_m - a_{m-1} \\
  
- \alpha^{\la -1}(c_{1}) - \alpha^{\la^{q-2}-1}(c_{1}) , - \alpha^{\la -1}(c_{2}) - \alpha^{\la^{q-2}-1}(c_{2}), \ldots, - \alpha^{\la -1}(c_{m-1}) - \alpha^{\la^{q-2}-1}(c_{m-1})  \\  
  
c_{1} + \alpha^{\la^{q-3} - \la^{q-2}}(\alpha^{\la -1}(c_{1})), c_{2} + \alpha^{\la^{q-3} - \la^{q-2}}(\alpha^{\la -1}(c_{2})), \ldots,  c_{m-1} + \alpha^{\la^{q-3} - \la^{q-2}}(\alpha^{\la-1}(c_{m-1})) \\  
  \hspace{10mm} \vdots \\
  
- \alpha^{\la -1}(c_{1}) - \alpha^{\la - \la^2}(c_{1}), - \alpha^{\la -1}(c_{2}) - \alpha^{\la - \la^2}(c_{2}), \ldots, - \alpha^{\la -1}(c_{m-1}) - \alpha^{\la - \la^2}(c_{m-1}) \\
 
 c_{2} +c_1 , c_3 + c_2, \ldots, c_{m-1} + c_{m-2} ,  a_{1} + \alpha^{q-1}(c_1).
 
\end{array}$$

The first two each contain the non-zero elements of~$A$ as~${\bf c}$ is a \#-harmonious sequence.   The third does because it comprises all of the rotational sequencing elements apart from~$a_1 - a_{m-1}$ and in its place we have $\alpha^{\la -1}(c_{m-1})  = a_1 - a_{m-1}$.  

That each of the fourth through to the penultimate sequences have the required elements follows from the properties of the automorphism~$\alpha$ and that~$c_1, c_2, \ldots, c_{m-1}$ are distinct and non-zero.

Finally, the last one is satisfied because we have all the sums of the \#-harmonious sequence~${\bf c}$, except $c_{m-1} + c_{1}$, and in its place we have
$a_{1} +   \alpha^{q-1}(c_1)  = c_1 + c_{m-1}$.
\end{proof}

Theorem~\ref{th:constr} is not the only way to successfully complete the template.  For example, in~\cite{Wang02} there is an alternative scheme for assigning the~$h_{ij}$ elements when~$A$ is cyclic of prime order and in~\cite{Ollis14} a different way of assigning the $g_i$ elements is used when~$q=3$.  These alternatives (and others) have potential for proving the sequenceability of more groups than we consider in the next section.

\section{The directed terraces}\label{sec:dt}

The target for this section is the proof of Theorem~\ref{th:spectrum}.   First we consider some structural properties of groups.  The following is well-known; see, for example,~\cite{PS00}.

\begin{lem}\label{lem:nonabnumbers}
Let $n = p_1^{a_1} \cdots p_t^{a_t}$, where the $p_i$ are distinct odd primes.
There is a nonabelian group of order~$n$ if and only if one of the following conditions applies:
\begin{enumerate}
\item $a_i \geq 3$ for some~$i$,
\item $n$ is cube-free and $p_i \equiv 1 \pmod{p_j}$ for some~$i,j$, 
\item $n$ is cube-free and $p_i^2 \equiv 1 \pmod{p_j}$ for some~$i,j$
\end{enumerate}
\end{lem}

We are able to meet all of these orders using groups of the form~$\Z_q \ltimes A$ for prime~$q$ and abelian~$A$.  When~$A \cong \Z_m$ there is the necessary automorphism of order~$q$ when either $q^2 | m$ or~$p | m$ for some prime~$p \equiv 1 \pmod{q}$.  These groups cover the first two cases of Lemma~\ref{lem:nonabnumbers}.  For the third case we use groups of the form~$A \cong \Z_p^2 \times B$ where~$p^2 \equiv 1 \pmod{q}$ and $p \nmid |B|$.  As it is little extra work, when~$3 \nmid |B|$ we consider groups of the form~$A \cong \Z_p^k \times B$ for any~$k \geq 2$.

We shall need an explicit construction for rotational sequencings for cyclic groups.  Define a {\em graceful permuation} of length~$k$ to be an arrangement $(g_1, g_2, \ldots, g_k)$ of the integers $\{ 1, 2, \ldots, k \}$ such that the absolute differences ${\bf h} = (h_1, h_2, \ldots, h_{k-1})$ given by $h_i = |g_{i+1} - g_i|$ are distinct.  This is equivalent to a graceful labeling of a path; graceful labelings of graphs are well-studied, see~\cite{GallianSurvey} for details.

For example, $(1,k,2,k-1,\ldots, \lfloor k/2 \rfloor + 1)$ is a graceful permutation of length~$k$ known as the {\em Walecki Construction}, see~\cite{Alspach08}.

Here is the connection between graceful permutations and rotational sequencings:

\begin{lem}\label{lem:g2rseq}{\rm \cite{FGM78}}
If $(g_1, g_2, \ldots, g_k)$ is a graceful permutation then $$(g_1, g_2, \ldots, g_k, g_k + k, g_{k-1} + k, \ldots, g_1 + k)$$ is a directed rotational terrace for~$\Z_{2k+1}$.
\end{lem}

We also need the following fact:
\begin{lem}\label{lem:g1}{\rm \cite{FFG83,Gvozdjak04}}
There is a graceful permutation~$(g_1, g_2, \ldots, g_k)$ with~$g_1 = x$ for each~$1 \leq x \leq k$.
\end{lem}

To extend rotational sequencings to non-cyclic groups we need a strengthening of the definition.  Let~${\bf a} = (a_1, a_2, \ldots, a_{m-1})$ be a directed R-terrace.  If $a_i  = a_{i-1} + a_{i+1}$ for some~$i$ then~${\bf a}$ is a {\em directed R$^*$-terrace} and its associated rotational sequencing is an {\em R$^*$-sequencing}.  If~$i=1$ then both the directed R$^*$-terrace and the R$^*$-sequencing are {\em standard}.  Any directed R$^*$-terrace may be made standard by re-indexing.

\begin{thm}\label{th:fgm}{\rm \cite{FGM78}}
Suppose $3 \nmid 2k+1$ and $A$ is an abelian group of odd order~$m$.
If $A$ is R$^*$-sequenceable then so is $A \times \Z_{2k+1}$.
\end{thm}

\begin{proof}[Proof Construction]
Let $(a_1, a_2, \ldots, a_{m-1})$ be a standard directed R$^*$-terrace for~$A$.

Exactly as in~\cite{FGM78} we list the sequences of first and second coordinates of a standard directed R$^*$-terrace for $A \times \Z_{2k+1}$ separately.

The first coordinates are given by $a_1, a_2, \ldots a_{m-1}$, followed by~$k$ copies of the length-$m$ sequence
$$ a_1, a_1, a_2, a_3, \ldots a_{m-1}$$
and then $k$ copies of the length-$m$ sequence
$$ 0, 0, a_2, a_3, \ldots a_{m-1}.$$

The second coordinates are given by $m-2$ 0s, followed by
$$
\begin{array}{l}
2k, 1, 2k ,1, \ldots, 2k, 1, 2 \\
2k-2, 2, 2k-2, 2, \ldots, 2k-2,2, 4 \\
\vdots \\
1, 2k, 1, 2k, \ldots, 1, 2k, 2r-1
\end{array}
$$
(where each line has $m$ elements starting with $(m-1)/2$ pairs of the form $x, -x$)
and then one final 0.
\end{proof}

The requirement that~$m$ is odd in Theorem~\ref{th:fgm} is unnecessary, but we do not require the construction for the even case.  We particularly need:

\begin{cor}\label{cor:fgm}{\rm \cite{FGM78}}
Let~$A$ and~$B$ be abelian groups of odd order.  Suppose that $A$ has a standard directed R$^*$-terrace $(a_1, \ldots, a_{m-1})$ and $3 \nmid |B|$. Then for any~$i$ with $1 \leq i \leq m-3$, the group $A \times B$ has a standard directed R$^*$-terrace.
\end{cor}

\begin{proof}
We may write~$B$ as the direct product of cyclic groups of odd order.  The result follows from repeated applications of the construction in the proof of Theorem~\ref{th:fgm}.
\end{proof}

Call the process of Theorem~\ref{th:fgm} and Corollary~\ref{cor:fgm} the {\em FGM Construction}.

To construct \#-harmonious sequences we need the closely-related notion of a harmonious sequence.  Let~${\bf c} = (c_1, c_2, \ldots, c_{m})$ be an arrangement of the elements of an abelian group~$A$ and define~${\bf d} = (d_1, d_2, \ldots, d_m)$ by $d_ i = c_i + c_{i+1}$ where the indices are calculated modulo~$m$ (so $d_{m} = c_m + c_1$).  If~${\bf d}$ is also an arrangement of the elements of~$A$ then ${\bf c}$ is a {\em harmonious sequence} and~$A$ is {\em harmonious}.

\begin{lem}\label{lem:hashharm}{\rm \cite{BGHJ91}}
Every odd-order abelian group except~$\Z_3$ is \#-harmonious.
\end{lem}

\begin{proof}[Proof Construction]
Let~$A$ be an abelian group of odd order other than~$\Z_3$ and write $$A = B \times \Z_{r_1} \times \Z_{r_2} \times \cdots \times \Z_{r_k}$$ where~$B$ is either $\Z_3^2$ or cyclic of order greater than~3.  The construction is by induction.  

We in fact show that these groups are {\em harmoniously matched}, meaning that they have  a harmonious sequence and a \#-harmonious sequence that start with the same element as each other and end with the same element as each other.  

To see that $\Z_3^2$ is harmoniously matched, consider the \#-harmonious sequence
$$ (1,1), (2,0), (2,1), (0,2), (2,2), (1,0), (1,2), (0,1)   $$
and the harmonious sequence
$$ (1,1) ,(2,1), (0,2), (1,2) , (2,2), (0,0), (1,0), (2,0), (0,1). $$
For~$r = 4\ell +1$ the \#-harmonious sequence
$$ 2\ell,  2\ell - 2, \ldots, 2, 4\ell, 4\ell-2, \ldots, 2\ell +2, 2\ell+1, 2\ell +3, \ldots, 4\ell-1, 1, 3, \ldots, 2\ell-1$$
is matched with the harmonious sequence
$$2\ell, 2\ell+1, \ldots, 4\ell, 0, 1, 2, \ldots, 2\ell -1.$$
For~$r = 4\ell +3$ with $\ell > 0$ the \#-harmonious sequence
$$2\ell +1, 2\ell -1, \ldots, 1, 4\ell +1, 4\ell -1, \ldots, 2\ell+3, 2\ell+2, 2\ell+4, \ldots, 4\ell+2, 2,3, \ldots, 2\ell$$
is matched with the harmonious sequence
$$2\ell+1, 2\ell+2, \ldots, 4\ell_2, 0, 1, 2, \ldots, 2\ell.$$
Note also that $(0,1,2)$ is a harmonious sequence for~$\Z_3$.

Now, let~$C$ and~$D$ be abelian groups of odd order with~$(c_1, c_2, \ldots, c_{m-1})$ and~$(c'_1, c'_2, \ldots, c'_m)$ a matched \#-harmonious and harmonious sequence for~$C$ and $(d_1, d_2, \ldots, d_n)$ a harmonious sequence for~$D$, indexed so that~$d_1= 0$.  Then
$$ (c_1, d_1),  (c_2, d_1), \ldots, (c_{m-1}, d_1); (c'_1, d_2), (c'_2, d_2), \ldots, (c'_m, d_2); 
\ldots ; (c'_1,d_n), \ldots, (c'_m, d_n)$$
is a \#-harmonious sequence for $C \times D$ and
$$ (c'_1, d_1),  (c'_2, d_1), \ldots, (c'_{m}, d_1); (c'_1, d_2), (c'_2, d_2), \ldots, (c'_m, d_2); 
\ldots ; (c'_1,d_n), \ldots, (c'_m, d_n)$$
is a harmonious sequence.  These may be re-indexed to be matched (for example, by taking endpoints $(c'_1, d_2)$ and $(c'_2, d_2)$).
\end{proof}

Call a \#-harmonious sequence constructed via the proof of Lemma~\ref{lem:hashharm} a {\em BGJH \#-harmonious sequence}.

We can now move to the main results.  First, we consider the case when~$A$ is cyclic.  

\begin{thm}\label{th:cyclic}
Let~$q$ be an odd prime.  Let~$A = \Z_m$, where~$m$ is odd and either $q^2 | m$ or~$p | m$ for some prime~$p \equiv 1 \pmod{q}$ and let~$\alpha$ be an automorphism of~$\Z_m$ of order~$q$.
Then~$G = \Z_{q} \ltimes_{\alpha} \Z_m$ is sequenceable. 
\end{thm}

\begin{proof}
We aim to satisfy the conditions of Theorem~\ref{th:constr}.
The arithmetical conditions on~$m$ ensure that we are considering exactly those~$m$ for which there is a value~$r$ such that multiplication by~$r$ is an automorphism of order~$q$ of~$\Z_m$.    Let~$\alpha$ be multiplication by~$r$ and let~$\lambda$ be a primitive root of~$q$ such that~$\la / (\la-1)$ is also a primitive root.  

Let~$m = 2k+1$.
The BGHJ \#-harmonious sequence for~$\Z_{m}$ has 1 and $-2$ as adjacent elements.   As~$\gcd(k+1,m) =1$, multiplication by~$r^{1 - \la}(k+1)$ is an automorphism of~$\Z_m$.  Apply this automorphism to the BGHJ \#-harmonious sequence to get one with $r^{1 - \la}(k+1)$ and $-2r^{1 - \la}(k+1)$ adjacent.  Index this \#-harmonious sequence so that $c_1 = -2r^{1 - \la}(k+1)$ and $c_{m-1} = r^{1 - \la}(k+1)$.

Suppose that~$c_1 + r^{1 -\la}(k+1) - r^{q-1}c_1 \neq 0$.  Let~$g_1$ be $c_1 + r^{1 -\la}(k+1) - r^{q-1}c_1 $ or $c_1 + r^{1 -\la}(k+1) - r^{q-1}c_1 - k$, whichever, when considered as an integer, is in the range~$1 \leq g_1 \leq k$.  Let ${\bf g}$ be a graceful permutation with first element~$g_1$ (which exists by Lemma~\ref{lem:g1}), and let~${\bf a} = (a_1, a_2, \ldots, a_{m-1})$ be either the directed rotational terrace constructed from~${\bf g}$ via Lemma~\ref{lem:g2rseq} or its reverse, so that $a_1 = c_1 + r^{1 -\la}(k+1) - r^{q-1}c_1$.

Now,  $a_1 = c_1 + c_{m-1} - \alpha^{q-1}(c_1)$ and, as $a_1 - a_{m-1} = k+1$, we have $a_{m-1} = a_1 - \alpha^{\la-1}(c_{m-1})$, and the conditions of Theorem~\ref{th:constr} are satisfied.

If $c_1 + r^{1 -\la}(k+1) - r^{q-1}c_1 = 0$, then reversing the initial roles of~1 and~$-2$ and running the same process gives alternative values for the variables such that $c_1 + r^{1 -\la}(k+1) - r^{q-1}c_1 \neq 0$ and the argument goes through as before.
\end{proof}

Next we move to non-cyclic $A$.     We shall need certain facts about the automorphism groups of abelian groups.  A comprehensive description of their structure can be found in~\cite{HR07}.  

Given two elements~$g$ and~$h$ of an abelian group, say that they are {\em independent} if $\langle g \rangle \cap \langle h \rangle = \{ 0 \}$.   Let $(g_1, h_1)$ and $(g_2, h_2)$ be pairs of independent elements in an abelian group~$A$.  If all four elements have the same prime order, then there is an automorphism of~$A$ that maps $g_1$ to $g_2$ and $h_1$ to $h_2$.

We shall be considering groups of the form $A = \Z_p^k \times B$ where $p$ is prime, $k \geq 2$ and $p \nmid |B|$.  If~$p^k \equiv 1 \pmod{q}$, for some prime~$q$ then~$A$ has an automorphism of order~$q$ with $\alpha \rest \Z_p^k$ also of order~$q$.  Further, we restrict our attention to~$\alpha$ of this form such that $\alpha \rest \Z_p^k$ (which is a subgroup of $\gl(k,p)$) is not diagonalisable, of which there is always at least one.

Suppose~$\alpha$ and $\beta$ are two automorphisms of~$A$. If $\alpha$ and $\beta$ are conjugate then $\Z_q \ltimes_{\alpha} A \cong \Z_q \ltimes_{\beta} A$.  We may therefore assume that~$\alpha \rest \Z_p^k$, which we think of as a matrix in $\gl(k,p)$ multiplying elements of~$A$ (considered as column vectors) from the left, is in {\em rational canonical form}. That is, it is a block-diagonal matrix where each block has 1s on the subdiagonal and 0s everywhere else except for the last column, see~\cite[Chapter~12]{DF04}.  In fact, as~$\alpha$ and~$\beta$ are conjugate if and only if $\alpha^c$ and~$\beta^c$  are conjugate for positive~$c$, we shall usually assume that a particular power of~$\alpha \rest \Z_p^k$ is in rational canonical form.

The first non-cyclic case we consider is when~$3 \nmid |A|$:

\begin{thm}\label{th:non3}
Let~$p$ and~$q$ be odd primes with~$p \neq 3$ and $p^k \equiv 1 \pmod{q}$.   Let~$B$ be an abelian group of odd order with~$3 \nmid |B|$.  Let~$\alpha$ be an automorphism of $\Z_p^k \times B$ of order~$q$ such that $\alpha \rest \Z_p^k$ is of order~$q$ and not diagonalisable.
Then $G = \Z_{q} \ltimes_{\alpha} (\Z_p^k \times B)$ is sequenceable. 
\end{thm}

\begin{proof}
The non-diagonalisable condition implies that $k \geq 2$.  As noted before the statement of the theorem, we may assume that some non-trivial power of~$\alpha \rest \Z_p^k$ is in rational canonical form with the additional condition that the $(1,1)$-entry is~0.  Let~$\la$ be a primitive root of~$q$ such that $\la/(\la - 1)$ is also a primitive root and assume that~$\alpha^{\la -1} \rest \Z_p^k$ is in this form. 

To apply Theorem~\ref{th:constr}, we require a directed R-terrace $(a_1, \ldots a_{m-1})$ and a \#-harmonious sequence $(c_1, \ldots, c_{m-1})$ such that $a_1 = c_1 + c_{m-1} - \alpha^{q-1}(c_1)$ and $a_{m-1} = a_1 - \alpha^{\la-1}(c_{m-1})$.  It is sufficient to find a directed R-terrace with $a_1$ and $a_{m-1}$ a pair of independent elements of prime order and a \#-harmonious sequence with $c_1 + c_{m-1} - \alpha^{q-1}(c_1)$ and $c_1 + c_{m-1} - \alpha^{q-1}(c_1) - \alpha^{\la-1}(c_{m-1})$ also a pair of independent elements of prime order.

The group~$\Z_p^k$ has a directed R$^*$-terrace~\cite{FGM78}.  Any arrangement of the elements of~$\Z_p^k \setminus \{ 0 \}$ unavoidably has many pairs of adjacent independent elements and hence any directed R$^*$-terrace does too.  Using this in the FGM construction gives a directed R$^*$-terrace with a pair of independent elements of order~$p$ which, by re-indexing, we may set to be~$a_1$ and~$a_{m-1}$.  

As noted in the proof of Theorem~\ref{th:cyclic}, the BGHJ \#-harmonious sequence for $\Z_p$ has $-2$ and $1$ as adjacent elements.  Further, using the BGHJ construction, we find that $\Z_p^k \times B$ has a \#-harmonious sequence with $(-2,0,\ldots,0)$ and $(1,0,\ldots,0)$ adjacent.    Our target is to show that the condition
$$\alpha^{\la-1}(c_{m-1}) \not\in \langle c_1 + c_{m-1} - \alpha^{q-1}(c_1) \rangle$$
holds, which implies the result we require.
We may re-index the \#-harmonious sequence so that $$\{ c_1, c_{m-1} \} = \{ (-2,0,\ldots,0),  (1,0,\ldots,0) \}$$ and there are two ways to do so.  

We have that $\alpha^{\la-1}((1,0, \ldots, 0)) = (0, 1, 0, \ldots, 0)$.  For the condition to fail in the case $c_{m-1} = (1,0, \ldots,0)$ we must have $\alpha^{q-1}((-2,0,\ldots,0)) = (1,x,0\ldots,0)$ for some~$x \in \Z_p$.  For the condition to fail in the case $c_{m-1} = (-2,0, \ldots,0)$ we must have $\alpha^{q-1}((1,0,\ldots,0)) = (1,y,0\ldots,0)$ for some~$y \in \Z_p$.  But as $\alpha^{q-1}((-2,0,\ldots,0)) = -2\alpha^{q-1}((1,0,\ldots,0))$, at least one of the two potential allocations of $c_1$ and $c_{m-1}$ does not violate the condition.
\end{proof}

Lastly we cover the cases where~$3 \mid |A|$ but~$27 \nmid |A|$.

\begin{thm}\label{th:3}
Let~$p$ and~$q$ be odd primes with~$p \neq 3$ and $p^2 \equiv 1 \pmod{q}$.   Let~$B$ be an abelian group of odd order with~$3 \nmid |B|$.  If~$\alpha$ is an automorphism of $\Z_p^2 \times \Z_3 \times B$ of order~$q$ such that $\alpha \rest \Z_p^2$ is of order~$q$ and not diagonalisable, then $G_1 = \Z_{q} \ltimes_{\alpha} (\Z_p^2 \times \Z_3 \times B)$ is sequenceable. 
 If~$\alpha$ is an automorphism of $\Z_p^2 \times \Z_9 \times B$ of order~$q$ such that $\alpha \rest \Z_p^2$ is of order~$q$ and not diagonalisable, then $G_2 = \Z_{q} \ltimes_{\alpha} (\Z_p^2 \times \Z_9 \times B)$ is sequenceable. 
\end{thm}

\begin{proof}
The structure of the proof is much the same as for Theorem~\ref{th:non3} in that we are looking for a directed R-terrace $(a_1, \ldots, a_{m-1})$ and a \#-harmonious sequence $(c_1, \ldots, c_{m-1})$ for $\Z_p^2 \times \Z_3 \times B$ (or $\Z_p^2 \times \Z_9 \times B$) such that $a_1$ and $a_{m-1}$ are a pair of independent elements of prime order~$p$ and $c_1 + c_{m-1} - \alpha^{q-1}(c_1)$ and $c_1 + c_{m-1} - \alpha^{q-1}(c_1) - \alpha^{\la-1}(c_{m-1})$ are also a pair of independent elements of order~$p$.

The argument for the existence of the latter of these is identical to that of Theorem~\ref{th:non3}.  The difficulty arises in the former case as the FGM construction for R$^*$-terraces does not directly apply.  In order to use the FGM construction we need to move the $\Z_3$ or $\Z_9$ factor into the base case and find directed R$^*$-terraces for $\Z_p^2 \times \Z_3$ and $\Z_p^2 \times \Z_9$ that have the independent pairs of adjacent elements we need.

First consider $\Z_p^2 \times \Z_3 \equiv \Z_{3p} \times \Z_p$.  The directed R-terrace $(a_1, \ldots a_{3p-1})$ for $\Z_{3p}$ constructed from the Walecki Construction using Lemma~\ref{lem:g2rseq} is in fact a directed R$^*$-terrace with~$a_{2p}$ the sum of its neighbors~\cite{FGM78}.  It also has $a_4 = 3(p-1)/2$ and $a_5 = 3$, which are elements of order~$p$.  Now, applying the construction of Theorem~\ref{th:fgm} we find a directed R$^*$-terrace of $ \Z_{3p} \times \Z_p$ with $(3(p-1)/2, 1)$ and $(3,-1 )$ adjacent, which are independent elements of order~$p$.  From here we can re-index and follow the method of Theorem~\ref{th:non3} to get the sequencing for~$G_1$.

Now consider $\Z_p^2 \times \Z_9 \equiv \Z_{9p} \times \Z_p$.  For this case we need part of the more complex constructions of~\cite{AKP17}.  However, as the aspects of that construction that we require are clearly stated within the paper, we do not recapitulate the full construction.  They use a device that they call ``the gadget"~\cite[Definition~3.5]{AKP17} which when used on a directed R$^*$-terrace of a group~$G$ and a group~$H$ that has orthomorphisms with an additional property, it produces a directed R$^*$-terrace for~$G \times H$.  Further, if $a_i = a_{i-1} + a_{i+1}$ gives the R$^*$ property in the directed R$^*$-terrace of~$G$ then we have $(a_i,0) = (a_{i-1},0) + (a_{i+1},0)$ doing the same for the resulting directed R$^*$-terrace of~$G \times H$~\cite[Lemma~3.7]{AKP17}.  It is allowable to take~$G = \Z_{p}$ and $H = \Z_{9}$ in this construction~\cite[Corollary~3.8]{AKP17}, which gives a standard directed R$^*$-terrace $(a_1, \ldots, a_{9p-1})$ for $\Z_{9p}$ with~$a_1$, $a_2$ and~$a_{9p-1}$ all of order~$p$.

We may now use the FGM Construction to give a directed R$^*$-terrace for~$\Z_{9p} \times \Z_p$.   This gives us adjacent elements of the form~$(a_1, 1)$ and $(a_1, -1)$, which are independent elements of order~$p$.  Re-index and follow the method of Theorem~\ref{th:non3} to get a sequencing for~$G_2$.
\end{proof}

Taken together with Lemma~\ref{lem:nonabnumbers}, Theorems~\ref{th:cyclic}, \ref{th:non3} and~\ref{th:3} prove Theorem~\ref{th:spectrum}.  While they also make some progress towards Keedwell's Conjecture, a full proof does not appear to be on the horizon.  Progress on the spectrum of not-necessarily-group-based complete Latin squares will require a totally different approach.

\section*{Acknowledgements}

The authors gratefully acknowledge the support of a Marlboro College Faculty Professional Development Grant.

\end{document}